\documentclass[12pt]{amsart}

\usepackage{pstricks}
\newpsobject{showgrid}{psgrid}{subgriddiv=1,griddots=5,gridlabels=6pt}
\usepackage{placeins}

\theoremstyle{definition}

\newtheorem{thm}[equation]{Theorem}
\newtheorem{lem}[equation]{Lemma}

\theoremstyle{remark}
\newtheorem*{rem}{Remark}

\begin{document}

\title{Tiling with Cuisenaire rods}
\author{M. Connolly}

\begin{abstract} In this paper a closed form expression for the number of tilings of an $n\times n$ square border with $1\times 1$ and $2\times1$ cuisenaire rods is proved using a transition matrix approach.  This problem is then generalised to $m\times n$ rectangular borders. The number of distinct tilings up to rotational symmetry is considered, and closed form expressions are given, in the case of a square border and in the case of a rectangular border. Finally, the number of distinct tilings up to dihedral symmetry is considered, and a closed form expression is given in the case of a square border.
\end{abstract}
\maketitle

\FloatBarrier
\section{Introduction to the problem}

Tilings by dominoes, or tiles of dimensions $2\times 1$, and $1\times 1$ square tiles have been widely studied. For example, it is famously true that the number of tilings of a $1\times n$ rectangle with dominoes and squares is equal to the $n+1$-th term of the Fibonacci sequence \cite{BenjaminQuinn}.

Furthermore, in a landmark paper Fisher and Temperley \cite{TemperleyFisher}, and independently in the same year Kastelyn \cite{Kastelyn}, prove an exact formula for the number of tilings of a $2m\times 2n$ rectangle by dominoes.

The main problem that will be discussed in this article is that of tiling a square border with cuisenaire rods of sizes $1$ and $2$. That is, we will use tiles of dimension $1\times 1$ and $2\times 1$ to build a square of size $n\times n$ with a square hole of size $n-2\times n-2$. Below are examples for $n=3$ and $n=4$.

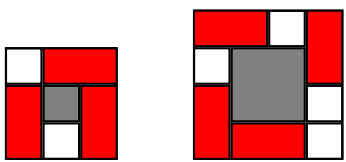
\begin{figure}[h]
\centering
\begin{pspicture}(0,0)(4.5,2) %\showgrid
%First example
\psframe[dimen=outer](0,0)(1.5,1.5)
\psframe[dimen=outer,fillstyle=solid,fillcolor=gray](0.5,0.5)(1,1)

\psframe[dimen=outer](0,1)(0.5,1.5)
\psframe[dimen=outer,fillstyle=solid,fillcolor=red](0.5,1)(1.5,1.5)
\psframe[dimen=outer,fillstyle=solid,fillcolor=red](1,0)(1.5,1)
\psframe[dimen=outer](0.5,0)(1,0.5)
\psframe[dimen=outer,fillstyle=solid,fillcolor=red](0,0)(0.5,1)

%Second Example
\psframe[dimen=outer](2.5,0)(4.5,2)
\psframe[dimen=outer,fillstyle=solid,fillcolor=gray](3,0.5)(4,1.5)
\psframe[dimen=outer,fillstyle=solid,fillcolor=red](2.5,1.5)(3.5,2)
\psframe[dimen=outer](3.5,1.5)(4,2)
\psframe[dimen=outer,fillstyle=solid,fillcolor=red](4,1)(4.5,2)
\psframe[dimen=outer](4,0.5)(4.5,1)
\psframe[dimen=outer](4,0)(4.5,0.5)
\psframe[dimen=outer,fillstyle=solid,fillcolor=red](3,0)(4,0.5)
\psframe[dimen=outer,fillstyle=solid,fillcolor=red](2.5,0)(3,1)
\psframe[dimen=outer](2.5,1)(3,1.5)

\end{pspicture}
\caption{Two examples of a square border tiling}
\label{fig:tiling0}
\end{figure}

 We explore how to calculate the number of such coverings for any natural number $n$. This problem first came to the author's attention in an article written for the Association of Teachers of Mathematics \cite{Zarzycki}. 

\pagebreak

\FloatBarrier

\section{Transition Matrix Approach}

First let us consider a tiling of a square border as a sequence of $1\times 1$ squares from the top left corner of the square clockwise around the border. We can imagine this sequence of squares laid out in a line from beginning to end as in figure \ref{fig:convention} below.

 As one traverses the border our orientation changes through one full revolution. To reflect the one-dimensional nature of the path we are only interested in closure to the front or back of the direction of travel. Hence square $1$ appears to be open to the bottom on the square border, but from the perspective of someone walking from $1$ to $8$, square number $1$ would be open to the rear of the direction of travel, and so appears open to the left when the squares are laid out end-to-end.

\begin{figure}[h]
\centering
\caption{Traversing the square border}
\begin{pspicture}(0,0)(10,2) %\showgrid
\psframe[dimen=outer](0,0)(1.5,1.5)
\psframe[dimen=outer,fillstyle=solid,fillcolor=gray](0.5,0.5)(1,1)

\psframe[dimen=outer,fillstyle=solid,fillcolor=red](0,0.5)(0.5,1.5)
\rput(0.25,1.25){$1$}
\psframe[dimen=outer](0.5,1)(1,1.5)
\psframe[dimen=outer,fillstyle=solid,fillcolor=red](1,0.5)(1.5,1.5)
\psframe[dimen=outer](1,0)(1.5,0.5)
\psframe[dimen=outer,fillstyle=solid,fillcolor=red](0,0)(1,0.5)

\rput(0.25,1.25){$1$}
\rput(0.75,1.25){$2$}
\rput(1.25,1.25){$3$}
\rput(1.25,0.75){$4$}
\rput(1.25,0.25){$5$}
\rput(0.75,0.25){$6$}
\rput(0.25,0.25){$7$}
\rput(0.25,0.75){$8$}

\psline(2.5,0)(3,0)
\psline(3,0)(3,0.5)
\psline(3,0.5)(2.5,0.5)

\psline(3.5,0)(4,0)
\psline(4,0)(4,0.5)
\psline(4,0.5)(3.5,0.5)
\psline(3.5,0)(3.5,0.5)

\psline(4.5,0)(5,0)
\psline(5,0.5)(4.5,0.5)
\psline(4.5,0)(4.5,0.5)

\psline(5.5,0)(6,0)
\psline(6,0)(6,0.5)
\psline(6,0.5)(5.5,0.5)

\psline(6.5,0)(7,0)
\psline(7,0)(7,0.5)
\psline(7,0.5)(6.5,0.5)
\psline(6.5,0)(6.5,0.5)

\psline(7.5,0)(8,0)
\psline(8,0.5)(7.5,0.5)
\psline(7.5,0)(7.5,0.5)

\psline(8.5,0)(9,0)
\psline(9,0)(9,0.5)
\psline(9,0.5)(8.5,0.5)

\psline(9.5,0)(10,0)
\psline(10,0.5)(9.5,0.5)
\psline(9.5,0)(9.5,0.5)

\rput(2.75,0.22){$1$}
\rput(3.75,0.22){$2$}
\rput(4.75,0.22){$3$}
\rput(5.75,0.22){$4$}
\rput(6.75,0.22){$5$}
\rput(7.75,0.22){$6$}
\rput(8.75,0.22){$7$}
\rput(9.75,0.22){$8$}
\end{pspicture}
\label{fig:convention}
\end{figure}
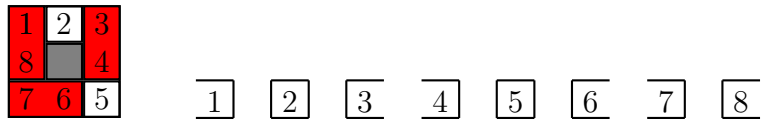

Using this convention one can identify three possible states for each square as you traverse the border. The first state is a closed $1\times 1$ square, the second is open on the left and the third open on the right. In the digraph of figure \ref{fig:transition} below, an arrow indicates that a state on the left of the arrow may precede a state to the right of the arrowhead as one traverses the border using the above convention.

\begin{figure}[h]
\centering
\caption{State transition}
\begin{pspicture}(0,0)(4,5) %\showgrid
\psframe[dimen=outer](0,4)(1,5)

\psline(0,2)(1,2)
\psline(1,2)(1,3)
\psline(1,3)(0,3)

\psline(0,0)(1,0)
\psline(1,1)(0,1)
\psline(0,0)(0,1)

\psframe[dimen=outer](3,4)(4,5)

\psline(3,2)(4,2)
\psline(4,2)(4,3)
\psline(4,3)(3,3)

\psline(3,0)(4,0)
\psline(4,1)(3,1)
\psline(3,0)(3,1)

\psline{->}(1.25,4.75)(2.75,4.75)
\psline{->}(1.25,4.5)(2.75,0.75)

\psline{->}(1.25,2.75)(2.75,4.5)
\psline{->}(1.25,2.5)(2.75,0.5)

\psline{->}(1.25,0.5)(2.75,2.5)

\end{pspicture}

\label{fig:transition}
\end{figure}
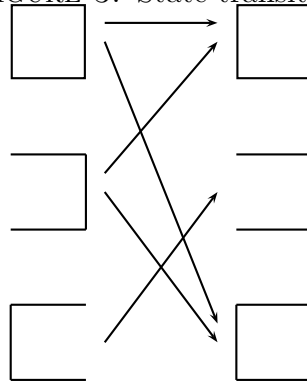

From this we can write down the state transition adjacency matrix as follows:

$$M=
\begin{bmatrix} 
1 & 0 & 1 \\
1 & 0 & 1 \\
0 & 1 & 0
\end{bmatrix}.$$

Now of course this is only for one step of the journey around the square. In our problem we are looking to enumerate certain paths in the following trellis graph of figure \ref{fig:trellis}.

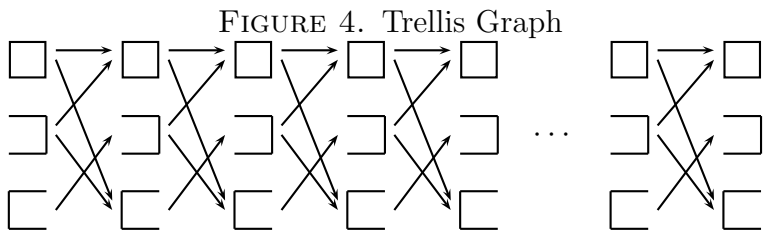
\begin{figure}[h]
\centering
\caption{Trellis Graph}
\begin{pspicture}(0,0)(10,2.5) %\showgrid
\psframe[dimen=outer](0,2)(0.5,2.5)

\psline(0,1)(0.5,1)
\psline(0.5,1)(0.5,1.5)
\psline(0.5,1.5)(0,1.5)

\psline(0,0)(0.5,0)
\psline(0.5,0.5)(0,0.5)
\psline(0,0)(0,0.5)

\psframe[dimen=outer](1.5,2)(2,2.5)

\psline(1.5,1)(2,1)
\psline(2,1)(2,1.5)
\psline(2,1.5)(1.5,1.5)

\psline(1.5,0)(2,0)
\psline(2,0.5)(1.5,0.5)
\psline(1.5,0)(1.5,0.5)

\psline{->}(0.625,2.375)(1.375,2.375)
\psline{->}(0.625,2.25)(1.375,0.375)

\psline{->}(0.625,1.375)(1.375,2.25)
\psline{->}(0.625,1.25)(1.375,0.25)

\psline{->}(0.625,0.25)(1.375,1.25)

\psframe[dimen=outer](3,2)(3.5,2.5)

\psline(3,1)(3.5,1)
\psline(3.5,1)(3.5,1.5)
\psline(3.5,1.5)(3,1.5)

\psline(3,0)(3.5,0)
\psline(3.5,0.5)(3,0.5)
\psline(3,0)(3,0.5)

\psline{->}(2.125,2.375)(2.875,2.375)
\psline{->}(2.125,2.25)(2.875,0.375)

\psline{->}(2.125,1.375)(2.875,2.25)
\psline{->}(2.125,1.25)(2.875,0.25)

\psline{->}(2.125,0.25)(2.875,1.25)

\psframe[dimen=outer](4.5,2)(5,2.5)

\psline(4.5,1)(5,1)
\psline(5,1)(5,1.5)
\psline(5,1.5)(4.5,1.5)

\psline(4.5,0)(5,0)
\psline(5,0.5)(4.5,0.5)
\psline(4.5,0)(4.5,0.5)

\psline{->}(3.625,2.375)(4.375,2.375)
\psline{->}(3.625,2.25)(4.375,0.375)

\psline{->}(3.625,1.375)(4.375,2.25)
\psline{->}(3.625,1.25)(4.375,0.25)

\psline{->}(3.625,0.25)(4.375,1.25)

\psframe[dimen=outer](6,2)(6.5,2.5)

\psline(6,1)(6.5,1)
\psline(6.5,1)(6.5,1.5)
\psline(6.5,1.5)(6,1.5)

\psline(6,0)(6.5,0)
\psline(6.5,0.5)(6,0.5)
\psline(6,0)(6,0.5)

\psline{->}(5.125,2.375)(5.875,2.375)
\psline{->}(5.125,2.25)(5.875,0.375)

\psline{->}(5.125,1.375)(5.875,2.25)
\psline{->}(5.125,1.25)(5.875,0.25)

\psline{->}(5.125,0.25)(5.875,1.25)

\rput(7.25,1.25){$\cdots$}

\psframe[dimen=outer](8,2)(8.5,2.5)

\psline(8,1)(8.5,1)
\psline(8.5,1)(8.5,1.5)
\psline(8.5,1.5)(8,1.5)

\psline(8,0)(8.5,0)
\psline(8.5,0.5)(8,0.5)
\psline(8,0)(8,0.5)

\psframe[dimen=outer](9.5,2)(10,2.5)

\psline(9.5,1)(10,1)
\psline(10,1)(10,1.5)
\psline(10,1.5)(9.5,1.5)

\psline(9.5,0)(10,0)
\psline(10,0.5)(9.5,0.5)
\psline(9.5,0)(9.5,0.5)

\psline{->}(8.625,2.375)(9.375,2.375)
\psline{->}(8.625,2.25)(9.375,0.375)

\psline{->}(8.625,1.375)(9.375,2.25)
\psline{->}(8.625,1.25)(9.375,0.25)

\psline{->}(8.625,0.25)(9.375,1.25)
\end{pspicture}
\label{fig:trellis}
\end{figure}

A tiling will result from a path; starting in the first state ending in the first or second states; starting in the second state and ending in the third; or starting in the third state and ending in the first or second states.

Our strategy, is to raise the adjacency matrix to the required power (one less than the number of squares traversed) and total the appropriate entries in the resulting matrix.

In the following section we detail some of the calculations required to find the matrix $M$ to any given power.

\FloatBarrier
\section{Linear Algebra Calculations}

In this section we diagonalise the matrix $M$ by solving the characteristic polynomial and changing to a basis of eigenvectors. It is then an elementary matter to raise $M$ to any power.

First we compute the characteristic polynomial.
\begin{align}
| M - xI| &= \begin{vmatrix}1-x & 0 & 1 \\ 1 & -x & 1 \\ 0 & 1 & -x \end{vmatrix} \notag \\
              &= -1 \begin{vmatrix} 1-x & 1 \\ 1 & 1 \end{vmatrix} -x \begin{vmatrix} 1-x & 0 \\ 1 & -x \end{vmatrix} \notag \\
              &= x(-x^{2}+x+1) \notag
\end{align}

The matrix $M$ has $3$ distinct eigenvalues $0$, $\frac{1+\sqrt{5}}{2}$ and $\frac{1-\sqrt{5}}{2}$.

From these we can find the corresponding eigenvectors,

$$ \begin{pmatrix} 1 \\ 0 \\ -1 \end{pmatrix}, \begin{pmatrix} \frac{1+\sqrt{5}}{2} \\ \frac{1+\sqrt{5}}{2} \\ 1 \end{pmatrix} , 
 \begin{pmatrix} \frac{1-\sqrt{5}}{2} \\ \frac{1-\sqrt{5}}{2} \\ 1 \end{pmatrix}.$$

Letting,

$$C=
\begin{bmatrix} 
1 & \frac{1+\sqrt{5}}{2} & \frac{1-\sqrt{5}}{2} \\
0 &  \frac{1+\sqrt{5}}{2} & \frac{1-\sqrt{5}}{2} \\
-1 & 1 & 1 
\end{bmatrix},
$$
it is routine to check that, 

$$C^{-1}=
\begin{bmatrix}
1 & -1 & 0 \\
\frac{5-\sqrt{5}}{10} & \frac{3\sqrt{5} - 5}{10} & \frac{5-\sqrt{5}}{10} \\
\frac{5+\sqrt{5}}{10} & -\frac{5+3\sqrt{5}}{10} & \frac{5+\sqrt{5}}{10}
\end{bmatrix}.
$$

Further that $D=C^{-1}MC$ is diagonal is a routine consequence. Conjugating by $C$, the last equation becomes $M=CDC^{-1}$.

Now we can compute $M$ to any power easily, as:

$$M^{q} = CD^{q}C^{-1}.$$

In the following section we apply these calculations to prove our main result.

\FloatBarrier
\section{Main Results}

\begin{lem}
\label{lem:count}
The number of $1\times 1$ squares in an $n\times n$ square border is equal to $4(n-1)$.
\end{lem}

\begin{proof}
The number of squares is equal to the difference of the outer square and the inner square.
\begin{align}
 & n^{2}-(n-2)^{2} \notag \\
& = n^{2}-(n^{2}- 4n+4) \notag \\
& = 4(n-1) \notag
\end{align}
\end{proof}

\begin{thm}
The number of tilings of a square border of size $n\times n$, where $n\geq 2$, by tiles of size $1\times 1$, $2\times 1$ or $1\times 2$ is equal to 
$$\left( \frac{1+\sqrt{5}}{2} \right)^{4(n-1)}+  \ \left( \frac{1-\sqrt{5}}{2} \right) ^{4(n-1)}$$
\end{thm}

\begin{proof}
Firstly note that the number of state transitions is one less than the number of squares in the border. By lemma \ref{lem:count} means that we need to compute $M^{4(n-1)-1}$. Let $$M^{4(n-1)-1}=:P=(p_{ij}).$$ 

We will enumerate all paths starting in state $1$ and ending in state $1$ or $2$, those starting in state $2$ and ending in state $3$, or those starting in state $3$ ending in states $1$ or $2$. We will do so by finding the sum,
$$p_{11}+p_{12}+p_{23}+p_{31}+p_{32}.$$

Now we compute the matrix $P$. Recall that $M =CDC^{-1}$, hence: 

$$P = CD^{4(n-1)-1}C^{-1}.$$

The latter product equals,

$$ \begin{bmatrix} 1 & \frac{1+\sqrt{5}}{2} & \frac{1-\sqrt{5}}{2} \\0 &  \frac{1+\sqrt{5}}{2} & \frac{1-\sqrt{5}}{2} \\-1 & 1 & 1 \end{bmatrix}
\text{diag}\begin{pmatrix} 0, & \left(\frac{1+\sqrt{5}}{2}\right)^{4(n-1)-1}, & \left(\frac{1-\sqrt{5}}{2}\right)^{4(n-1)-1} \end{pmatrix}
\begin{bmatrix}1 & -1 & 0 \\ \frac{5-\sqrt{5}}{10} & \frac{3\sqrt{5} - 5}{10} & \frac{5-\sqrt{5}}{10} \\ \frac{5+\sqrt{5}}{10} & -\frac{5+3\sqrt{5}}{10} & \frac{5+\sqrt{5}}{10}
\end{bmatrix}$$

$$ = \begin{bmatrix} 0 &  \left(\frac{1+\sqrt{5}}{2}\right)^{4(n-1)} & \left(\frac{1-\sqrt{5}}{2}\right)^{4(n-1)} \\ 0 &  \left(\frac{1+\sqrt{5}}{2}\right)^{4(n-1)} & \left(\frac{1-\sqrt{5}}{2}\right)^{4(n-1)} \\ 0 &  \left(\frac{1+\sqrt{5}}{2}\right)^{4(n-1)-1} & \left(\frac{1-\sqrt{5}}{2}\right)^{4(n-1)-1} \end{bmatrix}
\begin{bmatrix}1 & -1 & 0 \\ \left(\frac{5-\sqrt{5}}{10}\right) & \left(\frac{3\sqrt{5} - 5}{10}\right) & \left(\frac{5-\sqrt{5}}{10}\right) \\ \left(\frac{5+\sqrt{5}}{10}\right) & -\left(\frac{5+3\sqrt{5}}{10}\right) & \left(\frac{5+\sqrt{5}}{10}\right)
\end{bmatrix}$$

We now compute the entries of $P$.
\begin{align}
p_{11} &= \left(\frac{1+\sqrt{5}}{2}\right)^{4(n-1)}\left(\frac{5-\sqrt{5}}{10}\right)+\left(\frac{1-\sqrt{5}}{2}\right)^{4(n-1)}\left(\frac{5+\sqrt{5}}{10}\right), \notag \\
p_{12} &= \left(\frac{1+\sqrt{5}}{2}\right)^{4(n-1)}\left(\frac{3\sqrt{5}-5}{10}\right)-\left(\frac{1-\sqrt{5}}{2}\right)^{4(n-1)}\left(\frac{5+3\sqrt{5}}{10}\right), \notag \\
p_{23} &= \left(\frac{1+\sqrt{5}}{2}\right)^{4(n-1)}\left(\frac{5-\sqrt{5}}{10}\right)+\left(\frac{1-\sqrt{5}}{2}\right)^{4(n-1)}\left(\frac{5+\sqrt{5}}{10}\right), \notag \\
p_{31} &= \left(\frac{1+\sqrt{5}}{2}\right)^{4(n-1)-1}\left(\frac{5-\sqrt{5}}{10}\right)+\left(\frac{1-\sqrt{5}}{2}\right)^{4(n-1)-1}\left(\frac{5+\sqrt{5}}{10}\right), \notag \\
p_{32} &= \left(\frac{1+\sqrt{5}}{2}\right)^{4(n-1)-1}\left(\frac{3\sqrt{5}-5}{10}\right)-\left(\frac{1-\sqrt{5}}{2}\right)^{4(n-1)-1}\left(\frac{5+3\sqrt{5}}{10}\right). \notag
\end{align}

Finally, finding the sum of the entries above proves the claim.
\end{proof}

\begin{rem}
Checking in the $2\times 2$ case gives,
$$\left( \frac{1+\sqrt{5}}{2} \right)^{4(2-1)}+  \ \left( \frac{1-\sqrt{5}}{2} \right) ^{4(2-1)}=7.$$
Checking in the $3\times 3$ case gives,
$$\left( \frac{1+\sqrt{5}}{2} \right)^{4(3-1)}+  \ \left( \frac{1-\sqrt{5}}{2} \right) ^{4(3-1)}=47.$$
Both of these cases can be found in the Appendix in Section \ref{sec:appendix}.
\end{rem}

\begin{thm}
The number of tilings of a rectangular border of size $m \times n$, where $m,n\geq 2$, by tiles of size $1\times 1$, $2\times 1$ or $1\times 2$ is equal to 
$$\left( \frac{1+\sqrt{5}}{2} \right)^{2(m+n-2)}+  \ \left( \frac{1-\sqrt{5}}{2} \right) ^{2(m+n-2)}$$
\end{thm}

\begin{proof}
Firstly note that the number of squares in a rectangular border of size $m\times n$ is equal to 
\begin{align}
& mn-(m-2)(n-2) \notag \\
& = mn- (mn-2(m+n) + 4) \notag \\
& = 2(m+n-2). \notag
\end{align}
A similar computation with the matrix $D$ raised to the power $$2(m+n-2)-1$$ yields the result.
\end{proof}

\begin{rem}
It would be interesting to find similar formulae for the number of tilings of a square or rectangular border using other cuisenaire rods. A starting point might be to answer how many tilings of a $3\times 3$ square border are there using just $1\times 1$, $2\times 1$ and $3\times 1$ rods?
\end{rem}

\FloatBarrier
\section{Tilings up to symmetry}
 In general the orbits of this group action can be counted efficiently using Burnside's lemma, which states that the number of orbits is equal to the average number of fixed points,

$$|X/G| = \frac{1}{|G|} \sum_{g\in G} |X^{g}|.$$

Firstly we take into account only rotational symmetry in which case $G=C_{4}$, the cyclic group of order $4$.

\begin{thm}
\label{thm:rotational}
The number of tilings of an $n\times n$ square border, where $n\geq 2$ is an integer, by $1\times 1$ and $2\times 1$ cuisenaire rods, distinct up to rotational symmetry is given by the following formula

$$\frac{1}{4} \left[  x_{n}^{4} +(1+4(-1)^{n})x_{n}^{2}+2x_{n}+2(1+(-1)^{n}) \right].$$

Where $x_{n}=\phi^{n-1}+(-1)^{(n-1)}\phi^{-(n-1)}$ and $\phi = \frac{1+\sqrt{5}}{2}$.
\end{thm}

\begin{proof}
We apply Burnside's lemma. The factor $\frac{1}{4}$ comes from the fact that the group has order $4$. The identity fixes all square borders. A $90$-degree rotation, clockwise or anticlockwise, fixes those square borders where we are given choice over a quarter of the squares, which must match up from start to finish so that the count is the same as that for a square border with a quarter of the total number of squares. Finally, a $180$-degree rotation fixes all square borders where half the squares match the other half. One computes,

\begin{align}
 & \frac{1}{4}  \left[ \left( \frac{1+\sqrt{5}}{2} \right)^{4(n-1)}+\left( \frac{1-\sqrt{5}}{2} \right)^{4(n-1)} \right. + \notag \\
& +2\left( \left( \frac{1+\sqrt{5}}{2} \right)^{(n-1)}+\left( \frac{1-\sqrt{5}}{2} \right)^{(n-1)} \right) + \notag \\
&\left. + \left( \frac{1+\sqrt{5}}{2} \right)^{2(n-1)}+\left( \frac{1-\sqrt{5}}{2} \right)^{2(n-1)} \right] \notag
\end{align}

Substitution for $x_{n}$ as given in the statement of the theorem then yields the result.
\end{proof}

\begin{rem}
Theorem \ref{thm:rotational} for even $n$ becomes,

$$\frac{1}{4} \left[  x_{n}^{4} +5x_{n}^{2}+2x_{n}+4 \right] $$

and for odd $n$ becomes,

$$\frac{1}{4} \left[  x_{n}^{4} -3x_{n}^{2}+2x_{n} \right] .$$
\end{rem}

\pagebreak

\begin{thm}
The number of tilings of an $m\times n$ rectangular border, where $m,n\geq 2$ are integers, and $m\neq n$, by $1\times 1$ and $2\times 1$ cuisenaire rods, distinct up to rotational symmetry is given by the following formula

$$\frac{1}{2} \left[  x_{n}^{2} +x_{n} + 2(-1)^{m+n} \right].$$

Where $x_{n}=\phi^{m+n-1}+(-1)^{(m+n-1)}\phi^{-(m+n-1)}$ and $\phi = \frac{1+\sqrt{5}}{2}$.
\end{thm}
\begin{proof}
A similar computation using the rotational symmetry group of the rectangle which is $C_{2}$, the cyclic group of order $2$.
\end{proof}

Finally, taking into account reflections also, the symmetry group $G$ of the square acts on the set $X$ of all square border tilings. Now the group $G$ is the dihedral group of order $8$. 

It turns out that the cases of odd and even $n$ are different and we obtain two results.

\begin{thm}
The number of tilings of an $n\times n$ square border, where $n\geq 2$ is an even integer, by $1\times 1$ and $2\times 1$ cuisenaire rods, distinct up to rotations and reflections is given by the following formula

$$\frac{1}{8} \left[ (x_{n}^{2}+2)^{2} -2 + 2x_{n} +3(x_{n}^{2}+2) + \frac{6\sqrt{5}}{5}(\phi^{(n-1)} + \phi^{-(n-1)})x_{n} \right].$$

Where $x_{n}=\phi^{n-1}-\phi^{-(n-1)}$ and $\phi = \frac{1+\sqrt{5}}{2}$.
\end{thm}
\begin{proof} Omitted.
\end{proof}

\begin{thm}
The number of tilings of an $n\times n$ square border, where $n\geq 3$ is an odd integer, by $1\times 1$ and $2\times 1$ cuisenaire rods, distinct up to rotations and reflections is given by the following formula
$$\frac{1}{8} \left[ (x_{n}^{2}-2)^{2} -2 + 2x_{n} +(x_{n}^{2}-2) + \frac{4\sqrt{5}}{5}(\phi^{(n-1)} - \phi^{-(n-1)})x_{n} \right].$$
Where $x_{n}=\phi^{n-1}+\phi^{-(n-1)}$ and $\phi = \frac{1+\sqrt{5}}{2}$.
\end{thm}
\begin{proof} Omitted.
\end{proof}

\pagebreak

\FloatBarrier
\section{Appendix}
\label{sec:appendix}
This appendix features the full list of tilings in the $2\times 2$ and $3 \times 3$ cases. Note that tilings are not considered equivalent up to symmetry.
All possible tilings of a $2\times 2$ square with $1\times 1$ and $2 \times 1$ cuisenaire rods are illustrated in figure \ref{fig:tiling1}. Those tilings of a $3 \times 3$ square border are illustrated in figure \ref{fig:tiling2}.

\begin{figure}[h]
\centering
\caption{The $7$ tilings of a $2\times 2$ square}
\begin{pspicture}(0,0)(10,1) %\showgrid
\psframe[dimen=outer](0,0)(1,1)
\psframe[dimen=outer](0,0)(0.5,0.5)
\psframe[dimen=outer](0,0.5)(0.5,1)
\psframe[dimen=outer](0.5,0.5)(1,1)
\psframe[dimen=outer](0.5,0)(1,0.5)

\psframe[dimen=outer](1.5,0)(2.5,1)
\psframe[dimen=outer,fillstyle=solid,fillcolor=red](1.5,0.5)(2.5,1)
\psframe[dimen=outer](1.5,0)(2,0.5)
\psframe[dimen=outer](2,0)(2.5,0.5)

\psframe[dimen=outer](3,0)(4,1)
\psframe[dimen=outer](3,0)(3.5,0.5)
\psframe[dimen=outer](3,0.5)(3.5,1)
\psframe[dimen=outer,fillstyle=solid,fillcolor=red](3.5,0)(4,1)

\psframe[dimen=outer](4.5,0)(5.5,1)
\psframe[dimen=outer,fillstyle=solid,fillcolor=red](4.5,0)(5.5,0.5)
\psframe[dimen=outer](4.5,0.5)(5,1)
\psframe[dimen=outer](5,0.5)(5.5,1)

\psframe[dimen=outer](6,0)(7,1)
\psframe[dimen=outer](6.5,0)(7,0.5)
\psframe[dimen=outer](6.5,0.5)(7,1)
\psframe[dimen=outer,fillstyle=solid,fillcolor=red](6,0)(6.5,1)

\psframe[dimen=outer](7.5,0)(8.5,1)
\psframe[dimen=outer,fillstyle=solid,fillcolor=red](7.5,0)(8.5,0.5)
\psframe[dimen=outer,fillstyle=solid,fillcolor=red](7.5,0.5)(8.5,1)

\psframe[dimen=outer](9,0)(10,1)
\psframe[dimen=outer,fillstyle=solid,fillcolor=red](9,0)(9.5,1)
\psframe[dimen=outer,fillstyle=solid,fillcolor=red](9.5,0)(10,1)
\end{pspicture}
\label{fig:tiling1}
\end{figure}
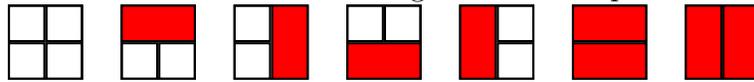

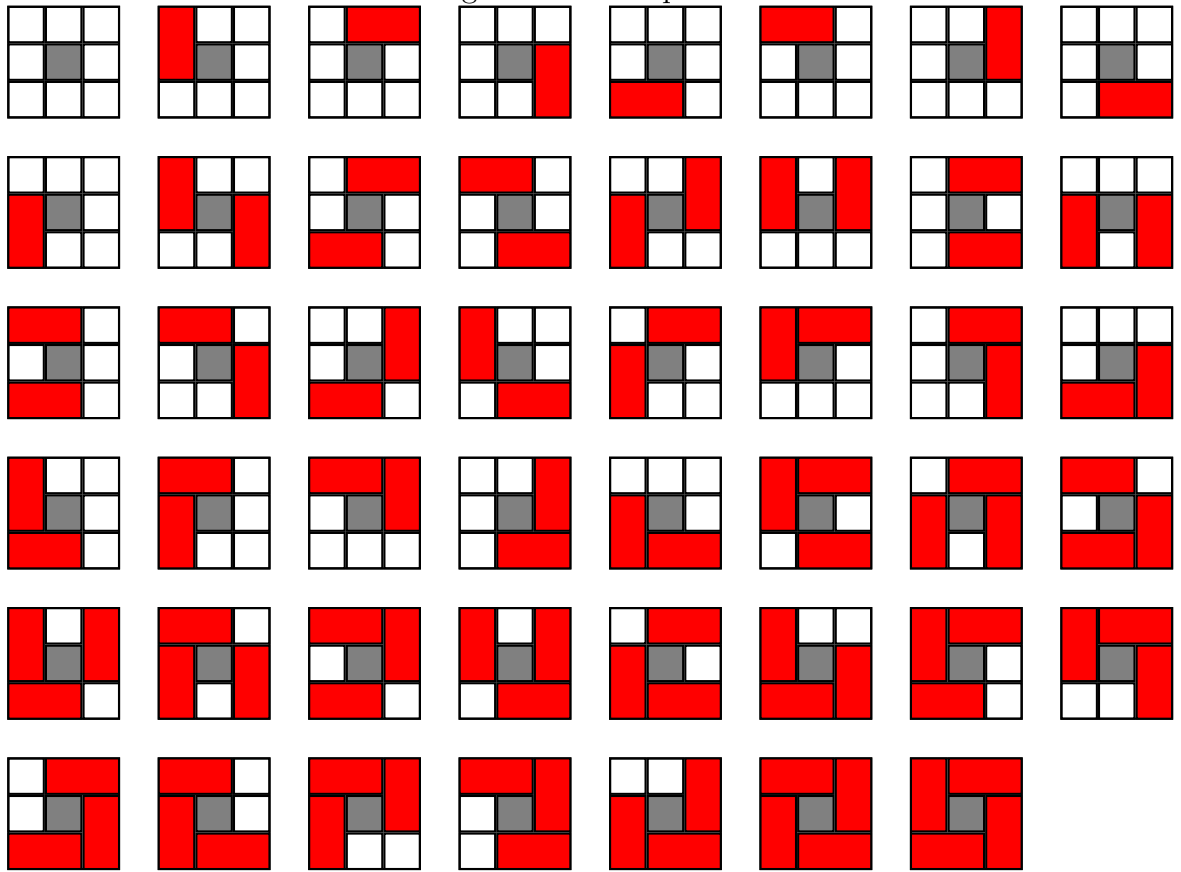
\begin{figure}[h]
\centering
\caption{The $47$ tilings of a $3\times 3$ square border}
\begin{pspicture}(0,0)(15.5,11.5) %\showgrid
%Row 1
%1
\psframe[dimen=outer](0,10)(1.5,11.5)
\psframe[dimen=outer,fillstyle=solid,fillcolor=gray](0.5,10.5)(1,11)
\psframe[dimen=outer](0,11)(0.5,11.5)
\psframe[dimen=outer](0,10.5)(0.5,11)
\psframe[dimen=outer](0,10)(0.5,10.5)

\psframe[dimen=outer](0.5,11)(1,11.5)
\psframe[dimen=outer](0.5,10)(1,10.5)

\psframe[dimen=outer](1,11)(1.5,11.5)
\psframe[dimen=outer](1,10.5)(1.5,11)
\psframe[dimen=outer](1,10)(1.5,10.5)
%2
\psframe[dimen=outer](2,10)(3.5,11.5)
\psframe[dimen=outer,fillstyle=solid,fillcolor=gray](2.5,10.5)(3,11)
\psframe[dimen=outer](3,10)(3.5,10.5)
\psframe[dimen=outer](3,10.5)(3.5,11)
\psframe[dimen=outer](3,11)(3.5,11.5)
\psframe[dimen=outer](2.5,11)(3,11.5)
\psframe[dimen=outer](2.5,10)(3,10.5)
\psframe[dimen=outer](2,10)(2.5,10.5)
\psframe[dimen=outer,fillstyle=solid,fillcolor=red](2,10.5)(2.5,11.5)

%3
\psframe[dimen=outer](4,10)(5.5,11.5)
\psframe[dimen=outer,fillstyle=solid,fillcolor=gray](4.5,10.5)(5,11)
\psframe[dimen=outer](4,10)(4.5,10.5)
\psframe[dimen=outer](4,10.5)(4.5,11)
\psframe[dimen=outer](4,11)(4.5,11.5)
\psframe[dimen=outer](5,10)(5.5,10.5)
\psframe[dimen=outer](4.5,10)(5,10.5)
\psframe[dimen=outer](5,10.5)(5.5,11)
\psframe[dimen=outer,fillstyle=solid,fillcolor=red](4.5,11)(5.5,11.5)

%4
\psframe[dimen=outer](6,10)(7.5,11.5)
\psframe[dimen=outer,fillstyle=solid,fillcolor=gray](6.5,10.5)(7,11)
\psframe[dimen=outer](6,10)(6.5,10.5)
\psframe[dimen=outer](6,10.5)(6.5,11)
\psframe[dimen=outer](6,11)(6.5,11.5)
\psframe[dimen=outer](6.5,11)(7,11.5)
\psframe[dimen=outer](7,11)(7.5,11.5)
\psframe[dimen=outer](6.5,10)(7,10.5)
\psframe[dimen=outer,fillstyle=solid,fillcolor=red](7,10)(7.5,11)

%5
\psframe[dimen=outer](8,10)(9.5,11.5)
\psframe[dimen=outer,fillstyle=solid,fillcolor=gray](8.5,10.5)(9,11)
\psframe[dimen=outer](8,11)(8.5,11.5)
\psframe[dimen=outer](8.5,11)(9,11.5)
\psframe[dimen=outer](9,11)(9.5,11.5)
\psframe[dimen=outer](8,10.5)(8.5,11)
\psframe[dimen=outer](9,10.5)(9.5,11)
\psframe[dimen=outer](9,10)(9.5,10.5)
\psframe[dimen=outer,fillstyle=solid,fillcolor=red](8,10)(9,10.5)

%6
\psframe[dimen=outer](10,10)(11.5,11.5)
\psframe[dimen=outer,fillstyle=solid,fillcolor=gray](10.5,10.5)(11,11)
\psframe[dimen=outer](10,10.5)(10.5,11)
\psframe[dimen=outer](10,10)(10.5,10.5)
\psframe[dimen=outer](10.5,10)(11,10.5)
\psframe[dimen=outer](11,11)(11.5,11.5)
\psframe[dimen=outer](11,10.5)(11.5,11)
\psframe[dimen=outer](11,10)(11.5,10.5)
\psframe[dimen=outer,fillstyle=solid,fillcolor=red](10,11)(11,11.5)

%7
\psframe[dimen=outer](12,10)(13.5,11.5)
\psframe[dimen=outer,fillstyle=solid,fillcolor=gray](12.5,10.5)(13,11)
\psframe[dimen=outer](12,10)(12.5,10.5)
\psframe[dimen=outer](12,10.5)(12.5,11)
\psframe[dimen=outer](12,11)(12.5,11.5)
\psframe[dimen=outer](12.5,11)(13,11.5)
\psframe[dimen=outer](12.5,10)(13,10.5)
\psframe[dimen=outer](13,10)(13.5,10.5)
\psframe[dimen=outer,fillstyle=solid,fillcolor=red](13,10.5)(13.5,11.5)

%8
\psframe[dimen=outer](14,10)(15.5,11.5)
\psframe[dimen=outer,fillstyle=solid,fillcolor=gray](14.5,10.5)(15,11)
\psframe[dimen=outer](14,10)(14.5,10.5)
\psframe[dimen=outer](14,10.5)(14.5,11)
\psframe[dimen=outer](14,11)(14.5,11.5)
\psframe[dimen=outer](14.5,11)(15,11.5)
\psframe[dimen=outer](15,11)(15.5,11.5)
\psframe[dimen=outer](15,10.5)(15.5,11)
\psframe[dimen=outer,fillstyle=solid,fillcolor=red](14.5,10)(15.5,10.5)

%Row 2

%9
\psframe[dimen=outer](0,8)(1.5,9.5)
\psframe[dimen=outer,fillstyle=solid,fillcolor=gray](0.5,8.5)(1,9)
\psframe[dimen=outer](0,9)(0.5,9.5)
\psframe[dimen=outer](0.5,9)(1,9.5)
\psframe[dimen=outer](0.5,8)(1,8.5)
\psframe[dimen=outer](1,9)(1.5,9.5)
\psframe[dimen=outer](1,8.5)(1.5,9)
\psframe[dimen=outer](1,8)(1.5,8.5)
\psframe[dimen=outer,fillstyle=solid,fillcolor=red](0,8)(0.5,9)

%10
\psframe[dimen=outer](2,8)(3.5,9.5)
\psframe[dimen=outer,fillstyle=solid,fillcolor=gray](2.5,8.5)(3,9)
\psframe[dimen=outer](2,8)(2.5,8.5)
\psframe[dimen=outer](2.5,9)(3,9.5)
\psframe[dimen=outer](2.5,8)(3,8.5)
\psframe[dimen=outer](3,9)(3.5,9.5)
\psframe[dimen=outer,fillstyle=solid,fillcolor=red](2,8.5)(2.5,9.5)
\psframe[dimen=outer,fillstyle=solid,fillcolor=red](3,8)(3.5,9)

%11
\psframe[dimen=outer](4,8)(5.5,9.5)
\psframe[dimen=outer,fillstyle=solid,fillcolor=gray](4.5,8.5)(5,9)
\psframe[dimen=outer](4,9)(4.5,9.5)
\psframe[dimen=outer](4,8.5)(4.5,9)
\psframe[dimen=outer](5,8.5)(5.5,9)
\psframe[dimen=outer](5,8)(5.5,8.5)
\psframe[dimen=outer,fillstyle=solid,fillcolor=red](4,8)(5,8.5)
\psframe[dimen=outer,fillstyle=solid,fillcolor=red](4.5,9)(5.5,9.5)

%12
\psframe[dimen=outer](6,8)(7.5,9.5)
\psframe[dimen=outer,fillstyle=solid,fillcolor=gray](6.5,8.5)(7,9)
\psframe[dimen=outer](7,9)(7.5,9.5)
\psframe[dimen=outer](7,8.5)(7.5,9)
\psframe[dimen=outer](6,8.5)(6.5,9)
\psframe[dimen=outer](6,8)(6.5,8.5)
\psframe[dimen=outer,fillstyle=solid,fillcolor=red](6.5,8)(7.5,8.5)
\psframe[dimen=outer,fillstyle=solid,fillcolor=red](6,9)(7,9.5)
%13
\psframe[dimen=outer](8,8)(9.5,9.5)
\psframe[dimen=outer,fillstyle=solid,fillcolor=gray](8.5,8.5)(9,9)
\psframe[dimen=outer](8,9)(8.5,9.5)
\psframe[dimen=outer](8.5,9)(9,9.5)
\psframe[dimen=outer](8.5,8)(9,8.5)
\psframe[dimen=outer](9,8)(9.5,8.5)
\psframe[dimen=outer,fillstyle=solid,fillcolor=red](8,8)(8.5,9)
\psframe[dimen=outer,fillstyle=solid,fillcolor=red](9,8.5)(9.5,9.5)

%14
\psframe[dimen=outer](10,8)(11.5,9.5)
\psframe[dimen=outer,fillstyle=solid,fillcolor=gray](10.5,8.5)(11,9)
\psframe[dimen=outer](10,8)(10.5,8.5)
\psframe[dimen=outer](10.5,9)(11,9.5)
\psframe[dimen=outer](10.5,8)(11,8.5)
\psframe[dimen=outer](11,8)(11.5,8.5)
\psframe[dimen=outer,fillstyle=solid,fillcolor=red](10,8.5)(10.5,9.5)
\psframe[dimen=outer,fillstyle=solid,fillcolor=red](11,8.5)(11.5,9.5)

%15
\psframe[dimen=outer](12,8)(13.5,9.5)
\psframe[dimen=outer,fillstyle=solid,fillcolor=gray](12.5,8.5)(13,9)
\psframe[dimen=outer](12,9)(12.5,9.5)
\psframe[dimen=outer](12,8.5)(12.5,9)
\psframe[dimen=outer](12,8)(12.5,8.5)
\psframe[dimen=outer](13,8.5)(13.5,9)
\psframe[dimen=outer,fillstyle=solid,fillcolor=red](12.5,8)(13.5,8.5)
\psframe[dimen=outer,fillstyle=solid,fillcolor=red](12.5,9)(13.5,9.5)

%16
\psframe[dimen=outer](14,8)(15.5,9.5)
\psframe[dimen=outer,fillstyle=solid,fillcolor=gray](14.5,8.5)(15,9)
\psframe[dimen=outer](14,9)(14.5,9.5)
\psframe[dimen=outer](14.5,9)(15,9.5)
\psframe[dimen=outer](14.5,8)(15,8.5)
\psframe[dimen=outer](15,9)(15.5,9.5)
\psframe[dimen=outer,fillstyle=solid,fillcolor=red](14,8)(14.5,9)
\psframe[dimen=outer,fillstyle=solid,fillcolor=red](15,8)(15.5,9)

%row3

%17
\psframe[dimen=outer](0,6)(1.5,7.5)
\psframe[dimen=outer,fillstyle=solid,fillcolor=gray](0.5,6.5)(1,7)
\psframe[dimen=outer](0,6.5)(0.5,7)
\psframe[dimen=outer](1,7)(1.5,7.5)
\psframe[dimen=outer](1,6.5)(1.5,7)
\psframe[dimen=outer](1,6)(1.5,6.5)
\psframe[dimen=outer,fillstyle=solid,fillcolor=red](0,6)(1,6.5)
\psframe[dimen=outer,fillstyle=solid,fillcolor=red](0,7)(1,7.5)
%18
\psframe[dimen=outer](2,6)(3.5,7.5)
\psframe[dimen=outer,fillstyle=solid,fillcolor=gray](2.5,6.5)(3,7)
\psframe[dimen=outer](2,6.5)(2.5,7)
\psframe[dimen=outer](2,6)(2.5,6.5)
\psframe[dimen=outer](2.5,6)(3,6.5)
\psframe[dimen=outer](3,7)(3.5,7.5)
%horizontal
\psframe[dimen=outer,fillstyle=solid,fillcolor=red](2,7)(3,7.5)
%vertical
\psframe[dimen=outer,fillstyle=solid,fillcolor=red](3,6)(3.5,7)
%19
\psframe[dimen=outer](4,6)(5.5,7.5)
\psframe[dimen=outer,fillstyle=solid,fillcolor=gray](4.5,6.5)(5,7)
\psframe[dimen=outer](4,7)(4.5,7.5)
\psframe[dimen=outer](4,6.5)(4.5,7)
\psframe[dimen=outer](4.5,7)(5,7.5)
\psframe[dimen=outer](5,6)(5.5,6.5)
\psframe[dimen=outer,fillstyle=solid,fillcolor=red](4,6)(5,6.5)
\psframe[dimen=outer,fillstyle=solid,fillcolor=red](5,6.5)(5.5,7.5)

%20
\psframe[dimen=outer](6,6)(7.5,7.5)
\psframe[dimen=outer,fillstyle=solid,fillcolor=gray](6.5,6.5)(7,7)
\psframe[dimen=outer](6,6)(6.5,6.5)
\psframe[dimen=outer](6.5,7)(7,7.5)
\psframe[dimen=outer](7,7)(7.5,7.5)
\psframe[dimen=outer](7,6.5)(7.5,7)
\psframe[dimen=outer,fillstyle=solid,fillcolor=red](6.5,6)(7.5,6.5)
\psframe[dimen=outer,fillstyle=solid,fillcolor=red](6,6.5)(6.5,7.5)

%21
\psframe[dimen=outer](8,6)(9.5,7.5)
\psframe[dimen=outer,fillstyle=solid,fillcolor=gray](8.5,6.5)(9,7)
\psframe[dimen=outer](8,7)(8.5,7.5)
\psframe[dimen=outer](8.5,6)(9,6.5)
\psframe[dimen=outer](9,6.5)(9.5,7)
\psframe[dimen=outer](9,6)(9.5,6.5)
\psframe[dimen=outer,fillstyle=solid,fillcolor=red](8.5,7)(9.5,7.5)
\psframe[dimen=outer,fillstyle=solid,fillcolor=red](8,6)(8.5,7)

%22
\psframe[dimen=outer](10,6)(11.5,7.5)
\psframe[dimen=outer,fillstyle=solid,fillcolor=gray](10.5,6.5)(11,7)
\psframe[dimen=outer](10,6)(10.5,6.5)
\psframe[dimen=outer](10.5,6)(11,6.5)
\psframe[dimen=outer](11,6.5)(11.5,7)
\psframe[dimen=outer](11,6)(11.5,6.5)
\psframe[dimen=outer,fillstyle=solid,fillcolor=red](10.5,7)(11.5,7.5)
\psframe[dimen=outer,fillstyle=solid,fillcolor=red](10,6.5)(10.5,7.5)

%23
\psframe[dimen=outer](12,6)(13.5,7.5)
\psframe[dimen=outer,fillstyle=solid,fillcolor=gray](12.5,6.5)(13,7)
\psframe[dimen=outer](12,7)(12.5,7.5)
\psframe[dimen=outer](12,6.5)(12.5,7)
\psframe[dimen=outer](12,6)(12.5,6.5)
\psframe[dimen=outer](12.5,6)(13,6.5)
\psframe[dimen=outer,fillstyle=solid,fillcolor=red](12.5,7)(13.5,7.5)
\psframe[dimen=outer,fillstyle=solid,fillcolor=red](13,6)(13.5,7)

%24
\psframe[dimen=outer](14,6)(15.5,7.5)
\psframe[dimen=outer,fillstyle=solid,fillcolor=gray](14.5,6.5)(15,7)
\psframe[dimen=outer](14,7)(14.5,7.5)
\psframe[dimen=outer](14,6.5)(14.5,7)
\psframe[dimen=outer](14.5,7)(15,7.5)
\psframe[dimen=outer](15,7)(15.5,7.5)
\psframe[dimen=outer,fillstyle=solid,fillcolor=red](14,6)(15,6.5)
\psframe[dimen=outer,fillstyle=solid,fillcolor=red](15,6)(15.5,7)

%Row 4
%25
\psframe[dimen=outer](0,4)(1.5,5.5)
\psframe[dimen=outer,fillstyle=solid,fillcolor=gray](0.5,4.5)(1,5)
\psframe[dimen=outer](1,5)(1.5,5.5)
\psframe[dimen=outer](1,4.5)(1.5,5)
\psframe[dimen=outer](1,4)(1.5,4.5)
\psframe[dimen=outer](0.5,5)(1,5.5)
\psframe[dimen=outer,fillstyle=solid,fillcolor=red](0,4)(1,4.5)
\psframe[dimen=outer,fillstyle=solid,fillcolor=red](0,4.5)(0.5,5.5)

%26
\psframe[dimen=outer](2,4)(3.5,5.5)
\psframe[dimen=outer,fillstyle=solid,fillcolor=gray](2.5,4.5)(3,5)
\psframe[dimen=outer](3,5)(3.5,5.5)
\psframe[dimen=outer](3,4.5)(3.5,5)
\psframe[dimen=outer](3,4)(3.5,4.5)
\psframe[dimen=outer](2.5,4)(3,4.5)
\psframe[dimen=outer,fillstyle=solid,fillcolor=red](2,5)(3,5.5)
\psframe[dimen=outer,fillstyle=solid,fillcolor=red](2,4)(2.5,5)

%27
\psframe[dimen=outer](4,4)(5.5,5.5)
\psframe[dimen=outer,fillstyle=solid,fillcolor=gray](4.5,4.5)(5,5)
\psframe[dimen=outer](4,4.5)(4.5,5)
\psframe[dimen=outer](4,4)(4.5,4.5)
\psframe[dimen=outer](4.5,4)(5,4.5)
\psframe[dimen=outer](5,4)(5.5,4.5)
\psframe[dimen=outer,fillstyle=solid,fillcolor=red](4,5)(5,5.5)
\psframe[dimen=outer,fillstyle=solid,fillcolor=red](5,4.5)(5.5,5.5)

%28
\psframe[dimen=outer](6,4)(7.5,5.5)
\psframe[dimen=outer,fillstyle=solid,fillcolor=gray](6.5,4.5)(7,5)
\psframe[dimen=outer](6,5)(6.5,5.5)
\psframe[dimen=outer](6,4.5)(6.5,5)
\psframe[dimen=outer](6,4)(6.5,4.5)
\psframe[dimen=outer](6.5,5)(7,5.5)
\psframe[dimen=outer,fillstyle=solid,fillcolor=red](6.5,4)(7.5,4.5)
\psframe[dimen=outer,fillstyle=solid,fillcolor=red](7,4.5)(7.5,5.5)

%29
\psframe[dimen=outer](8,4)(9.5,5.5)
\psframe[dimen=outer,fillstyle=solid,fillcolor=gray](8.5,4.5)(9,5)
\psframe[dimen=outer](8,5)(8.5,5.5)
\psframe[dimen=outer](8.5,5)(9,5.5)
\psframe[dimen=outer](9,5)(9.5,5.5)
\psframe[dimen=outer](9,4.5)(9.5,5)
\psframe[dimen=outer,fillstyle=solid,fillcolor=red](8.5,4)(9.5,4.5)
\psframe[dimen=outer,fillstyle=solid,fillcolor=red](8,4)(8.5,5)

%30
\psframe[dimen=outer](10,4)(11.5,5.5)
\psframe[dimen=outer,fillstyle=solid,fillcolor=gray](10.5,4.5)(11,5)
\psframe[dimen=outer](10,4)(10.5,4.5)
\psframe[dimen=outer](11,4.5)(11.5,5)
\psframe[dimen=outer,fillstyle=solid,fillcolor=red](10.5,4)(11.5,4.5)
\psframe[dimen=outer,fillstyle=solid,fillcolor=red](10.5,5)(11.5,5.5)
\psframe[dimen=outer,fillstyle=solid,fillcolor=red](10,4.5)(10.5,5.5)

%31
\psframe[dimen=outer](12,4)(13.5,5.5)
\psframe[dimen=outer,fillstyle=solid,fillcolor=gray](12.5,4.5)(13,5)
\psframe[dimen=outer](12,5)(12.5,5.5)
\psframe[dimen=outer](12.5,4)(13,4.5)
\psframe[dimen=outer,fillstyle=solid,fillcolor=red](12,4)(12.5,5)
\psframe[dimen=outer,fillstyle=solid,fillcolor=red](13,4)(13.5,5)
\psframe[dimen=outer,fillstyle=solid,fillcolor=red](12.5,5)(13.5,5.5)

%32
\psframe[dimen=outer](14,4)(15.5,5.5)
\psframe[dimen=outer,fillstyle=solid,fillcolor=gray](14.5,4.5)(15,5)
\psframe[dimen=outer](14,4.5)(14.5,5)
\psframe[dimen=outer](15,5)(15.5,5.5)
\psframe[dimen=outer,fillstyle=solid,fillcolor=red](14,4)(15,4.5)
\psframe[dimen=outer,fillstyle=solid,fillcolor=red](14,5)(15,5.5)
\psframe[dimen=outer,fillstyle=solid,fillcolor=red](15,4)(15.5,5)

%Row 5
%33
\psframe[dimen=outer](0,2)(1.5,3.5)
\psframe[dimen=outer,fillstyle=solid,fillcolor=gray](0.5,2.5)(1,3)
\psframe[dimen=outer](0.5,3)(1,3.5)
\psframe[dimen=outer](1,2)(1.5,2.5)
\psframe[dimen=outer,fillstyle=solid,fillcolor=red](0,2.5)(0.5,3.5)
\psframe[dimen=outer,fillstyle=solid,fillcolor=red](1,2.5)(1.5,3.5)
\psframe[dimen=outer,fillstyle=solid,fillcolor=red](0,2)(1,2.5)

%34
\psframe[dimen=outer](2,2)(3.5,3.5)
\psframe[dimen=outer,fillstyle=solid,fillcolor=gray](2.5,2.5)(3,3)
\psframe[dimen=outer](3,3)(3.5,3.5)
\psframe[dimen=outer](2.5,2)(3,2.5)
\psframe[dimen=outer,fillstyle=solid,fillcolor=red](2,2)(2.5,3)
\psframe[dimen=outer,fillstyle=solid,fillcolor=red](3,2)(3.5,3)
\psframe[dimen=outer,fillstyle=solid,fillcolor=red](2,3)(3,3.5)

%\psframe[dimen=outer](0,3)(0.5,3.5)
%\psframe[dimen=outer](0,2.5)(0.5,3)
%\psframe[dimen=outer](0,2)(0.5,2.5)
%
%\psframe[dimen=outer](0.5,3)(1,5.3)
%\psframe[dimen=outer](0.5,2)(1,2.5)
%
%\psframe[dimen=outer](1,3)(1.5,3.5)
%\psframe[dimen=outer](1,2.5)(1.5,3)
%\psframe[dimen=outer](1,2)(1.5,2.5)

%35
\psframe[dimen=outer](4,2)(5.5,3.5)
\psframe[dimen=outer,fillstyle=solid,fillcolor=gray](4.5,2.5)(5,3)
\psframe[dimen=outer,fillstyle=solid,fillcolor=red](4,3)(5,3.5)
\psframe[dimen=outer,fillstyle=solid,fillcolor=red](4,2)(5,2.5)
\psframe[dimen=outer,fillstyle=solid,fillcolor=red](5,2.5)(5.5,3.5)
\psframe[dimen=outer](4,2.5)(4.5,3)
\psframe[dimen=outer](5,2)(5.5,2.5)

%36
\psframe[dimen=outer](6,2)(7.5,3.5)
\psframe[dimen=outer,fillstyle=solid,fillcolor=gray](6.5,2.5)(7,3)
\psframe[dimen=outer](6.5,3)(7,3.5)
\psframe[dimen=outer](6,2)(6.5,2.5)
\psframe[dimen=outer,fillstyle=solid,fillcolor=red](6,2.5)(6.5,3.5)
\psframe[dimen=outer,fillstyle=solid,fillcolor=red](7,2.5)(7.5,3.5)
\psframe[dimen=outer,fillstyle=solid,fillcolor=red](6.5,2)(7.5,2.5)

%37
\psframe[dimen=outer](8,2)(9.5,3.5)
\psframe[dimen=outer,fillstyle=solid,fillcolor=gray](8.5,2.5)(9,3)
\psframe[dimen=outer](8,3)(8.5,3.5)
\psframe[dimen=outer](9,2.5)(9.5,3)
\psframe[dimen=outer,fillstyle=solid,fillcolor=red](8.5,2)(9.5,2.5)
\psframe[dimen=outer,fillstyle=solid,fillcolor=red](8.5,3)(9.5,3.5)
\psframe[dimen=outer,fillstyle=solid,fillcolor=red](8,2)(8.5,3)
%38
\psframe[dimen=outer](10,2)(11.5,3.5)
\psframe[dimen=outer,fillstyle=solid,fillcolor=gray](10.5,2.5)(11,3)
\psframe[dimen=outer](10.5,3)(11,3.5)
\psframe[dimen=outer](11,3)(11.5,3.5)
\psframe[dimen=outer,fillstyle=solid,fillcolor=red](11,2)(11.5,3)
\psframe[dimen=outer,fillstyle=solid,fillcolor=red](10,2)(11,2.5)
\psframe[dimen=outer,fillstyle=solid,fillcolor=red](10,2.5)(10.5,3.5)

%39
\psframe[dimen=outer](12,2)(13.5,3.5)
\psframe[dimen=outer,fillstyle=solid,fillcolor=gray](12.5,2.5)(13,3)
\psframe[dimen=outer](13,2.5)(13.5,3)
\psframe[dimen=outer](13,2)(13.5,2.5)
\psframe[dimen=outer,fillstyle=solid,fillcolor=red](12.5,3)(13.5,3.5)
\psframe[dimen=outer,fillstyle=solid,fillcolor=red](12,2.5)(12.5,3.5)
\psframe[dimen=outer,fillstyle=solid,fillcolor=red](12,2)(13,2.5)

%40
\psframe[dimen=outer](14,2)(15.5,3.5)
\psframe[dimen=outer,fillstyle=solid,fillcolor=gray](14.5,2.5)(15,3)
\psframe[dimen=outer,fillstyle=solid,fillcolor=red](15,2)(15.5,3)
\psframe[dimen=outer,fillstyle=solid,fillcolor=red](14,2.5)(14.5,3.5)
\psframe[dimen=outer,fillstyle=solid,fillcolor=red](14.5,3)(15.5,3.5)
\psframe[dimen=outer](14,2)(14.5,2.5)
\psframe[dimen=outer](14.5,2)(15,2.5)
%Row 6
%41
\psframe[dimen=outer](0,0)(1.5,1.5)
\psframe[dimen=outer,fillstyle=solid,fillcolor=gray](0.5,0.5)(1,1)
\psframe[dimen=outer](0,0.5)(0.5,1)
\psframe[dimen=outer](0,1)(0.5,1.5)
\psframe[dimen=outer,fillstyle=solid,fillcolor=red](0.5,1)(1.5,1.5)
\psframe[dimen=outer,fillstyle=solid,fillcolor=red](0,0)(1,0.5)
\psframe[dimen=outer,fillstyle=solid,fillcolor=red](1,0)(1.5,1)
%42
\psframe[dimen=outer](2,0)(3.5,1.5)
\psframe[dimen=outer,fillstyle=solid,fillcolor=gray](2.5,0.5)(3,1)
\psframe[dimen=outer](3,0.5)(3.5,1)
\psframe[dimen=outer](3,1)(3.5,1.5)
\psframe[dimen=outer,fillstyle=solid,fillcolor=red](2,1)(3,1.5)
\psframe[dimen=outer,fillstyle=solid,fillcolor=red](2,0)(2.5,1)
\psframe[dimen=outer,fillstyle=solid,fillcolor=red](2.5,0)(3.5,0.5)
%43
\psframe[dimen=outer](4,0)(5.5,1.5)
\psframe[dimen=outer,fillstyle=solid,fillcolor=gray](4.5,0.5)(5,1)
\psframe[dimen=outer](4.5,0)(5,0.5)
\psframe[dimen=outer](5,0)(5.5,0.5)
\psframe[dimen=outer,fillstyle=solid,fillcolor=red](4,0)(4.5,1)
\psframe[dimen=outer,fillstyle=solid,fillcolor=red](5,0.5)(5.5,1.5)
\psframe[dimen=outer,fillstyle=solid,fillcolor=red](4,1)(5,1.5)
%44
\psframe[dimen=outer](6,0)(7.5,1.5)
\psframe[dimen=outer,fillstyle=solid,fillcolor=gray](6.5,0.5)(7,1)
\psframe[dimen=outer](6,0)(6.5,0.5)
\psframe[dimen=outer](6,0.5)(6.5,1)
\psframe[dimen=outer,fillstyle=solid,fillcolor=red](6,1)(7,1.5)
\psframe[dimen=outer,fillstyle=solid,fillcolor=red](7,0.5)(7.5,1.5)
\psframe[dimen=outer,fillstyle=solid,fillcolor=red](6.5,0)(7.5,0.5)
%45
\psframe[dimen=outer](8,0)(9.5,1.5)
\psframe[dimen=outer,fillstyle=solid,fillcolor=gray](8.5,0.5)(9,1)
\psframe[dimen=outer](8,1)(8.5,1.5)
\psframe[dimen=outer](8.5,1)(9,1.5)
\psframe[dimen=outer,fillstyle=solid,fillcolor=red](8.5,0)(9.5,0.5)
\psframe[dimen=outer,fillstyle=solid,fillcolor=red](9,0.5)(9.5,1.5)
\psframe[dimen=outer,fillstyle=solid,fillcolor=red](8,0)(8.5,1)
%46
\psframe[dimen=outer](10,0)(11.5,1.5)
\psframe[dimen=outer,fillstyle=solid,fillcolor=gray](10.5,0.5)(11,1)
\psframe[dimen=outer,fillstyle=solid,fillcolor=red](10.5,0)(11.5,0.5)
\psframe[dimen=outer,fillstyle=solid,fillcolor=red](11,0.5)(11.5,1.5)
\psframe[dimen=outer,fillstyle=solid,fillcolor=red](10,0)(10.5,1)
\psframe[dimen=outer,fillstyle=solid,fillcolor=red](10,1)(11,1.5)
%47
\psframe[dimen=outer](12,0)(13.5,1.5)
\psframe[dimen=outer,fillstyle=solid,fillcolor=gray](12.5,0.5)(13,1)
\psframe[dimen=outer,fillstyle=solid,fillcolor=red](12,0.5)(12.5,1.5)
\psframe[dimen=outer,fillstyle=solid,fillcolor=red](13,0)(13.5,1)
\psframe[dimen=outer,fillstyle=solid,fillcolor=red](12.5,1)(13.5,1.5)
\psframe[dimen=outer,fillstyle=solid,fillcolor=red](12,0)(13,0.5)
\end{pspicture}
\label{fig:tiling2}
\end{figure}

\FloatBarrier

\end{document}